\documentclass{amsart}

\usepackage[utf8]{inputenc}
\usepackage{todonotes}
\usepackage{amsmath,amssymb,amsthm,units, stmaryrd,stackrel,relsize,bm}
\usepackage{amsmath,amssymb,units, stmaryrd,stackrel,relsize,bm}
\usepackage{enumitem}
\usepackage{mathdots}
\usepackage[T1]{fontenc}
\usepackage{accents}
\usepackage{hyperref}
\usepackage{tikz,tikz-cd}
\usetikzlibrary{calc,arrows, decorations.pathmorphing, matrix,fit}
\usepackage{float}
\newtheorem{theorem}{Theorem}

\newtheorem{lemma}[theorem]{Lemma}

\newtheorem{claim}[theorem]{Claim}

\newtheorem{corollary}[theorem]{Corollary}

\theoremstyle{definition}
\newtheorem{definition}[theorem]{Definition}

\theoremstyle{remark}

\newcommand\Ord{{\mathsf{Ord}}}

\newcommand\glp{\mathsf{GLP}}
\newcommand\gl{\mathsf{GL}}

\newcommand{\tup}[1]{\langle #1\rangle}
\newcommand{\set}[1]{\{ #1\}}
\newcommand{\sub}[1]{\mathrm{sub}(#1)}
\newcommand{\vars}{\mathsf{Vars}}
\newcommand{\inter}[1]{\llbracket#1\rrbracket}

\newcommand{\lang}{\mathcal{L}}

\makeatletter
\def\Ddots{\mathinner{\mkern1mu\raise\p@
\vbox{\kern7\p@\hbox{.}}\mkern2mu
\raise4\p@\hbox{.}\mkern2mu\raise7\p@\hbox{.}\mkern1mu}}
\makeatother
\let\temp\phi
\let\phi\varphi
\let\varphi\temp

\let\temp\epsilon
\let\epsilon\varepsilon
\let\varepsilon\temp

\begin{document}

%\begin{frontmatter}
  \title{Strong completeness of the logic J}
  \author{J. P. Aguilera}
  \address{J. P. A. Institute of Discrete Mathematics and Geometry, Vienna University of Technology. \\\
  Wiedner Hauptstra{\ss}e 8--10, 1040 Vienna, Austria.}
  \email{aguilera@logic.at}
  \author{G. Stepanov}
\address{G. S. Institute of Discrete Mathematics and Geometry, Vienna University of Technology. \\\
  Wiedner Hauptstra{\ss}e 8--10, 1040 Vienna, Austria.}
  \email{grigorii.stepanov@tuwien.ac.at}

\begin{abstract}
  We prove that the polymodal logic $\mathsf{J}$ is strongly complete with respect to \textit{$\mathsf{J}$-bouquets}, a topological refinement of its Kripke semantics. 
In particular, it is strongly topologically complete. This yields the following completeness result for the provability logic $\glp$: a countable set of formul\ae\  $\Gamma$ is consistent with $\glp$ if and only if there is a $\mathsf{J}$-bouquet $B$ and $r\in B$ such that $B, r\Vdash \glp$ and $B, r\Vdash\Gamma$.
In contrast, 
we exhibit counterexamples showing that $\glp$ is not strongly complete with respect to Beklemishev-Gabelaia spaces.
\end{abstract}

\keywords{
  Provability logic, strong completeness, Ignatiev frame, GLP}
  \date{\today}
  \maketitle
%\end{frontmatter}

%\begin{document}

\numberwithin{equation}{section}
\setcounter{tocdepth}{1}
\tableofcontents

\section{Introduction}
 The interest in provability logic stems from the investigations of Gödel's incompleteness theorems.
 L\"ob \cite{loeb} formulated three conditions on the provability predicate of Peano Arithmetic
 that form a useful modification of the  conditions that Hilbert and Bernays \cite{hilber} introduced for their proof of Gödel’s second incompleteness theorem. 
 Friedman \cite{Fr75} posed the problem of axiomatizing the set of valid arithmetical formul\ae\  built from expressions of the form ``$\phi$ is provable'' by means of Boolean connectives and provability assertions.  Building on work of Segerberg \cite{Se71} on the Kripke semantics of L\"ob's logic $\mathsf{GL}$, Boolos \cite{Boo75} (and independently Bernardi, Montagna, and van Benthem) proved that L\"ob's axiomatization was arithmetically complete when restricting to closed (i.e., variable-free) formul\ae. Solovay \cite{Solovay1976} later extended Boolos' theorem to a completeness theorem of  $\mathsf{GL}$ for its arithmetical interpretation.
 
We consider a propositional modal logic with infinitely many modalities ($[n]$ and its dual $\tup{n}$ for $n<\omega$). We call this language $\mathcal{L}$ and its restriction to the first $n$ modalities $\mathcal{L}_n$.
The intended interpretation of $[n]\phi$ is ``$\phi$ is provable in the formal system $T_n$,'' where $T_n$ is an arithmetical or set-theoretic system such that $T_{n+1}$ is stronger than $T_n$ for each $n$. We shall not concern ourselves with the arithmetical interpretation of this polymodal provability logic here, however, and instead focus on its relational and topological semantics (see below). This polymodal extension of $\gl$ -- called $\glp$ -- was introduced by Japaridze \cite{Ja88} and is complete with respect to various choices of $T_n$; see e.g., Japaridze \cite{Ja88} or Fernandez-Duque and Joosten \cite{FDJo18}. Beklemishev \cite{BekWorms} and Beklemishev and Pakhomov \cite{BePa22} show how $\glp$ can be applied for the purposes of ordinal analysis and other proof-theoretic results. The first author and Pakhomov \cite{AgPa24} have also proved the completeness of the extension $\glp.3$ of $\glp$ for a set-theoretic interpretation, extending a result of Solovay \cite{Solovay1976} for the unimodal case.

Much work has been carried out on the models of $\gl$ and $\glp$, in part due to the fact that these are generally complicated. For instance, it is not difficult to show that $\glp$ has no nontrivial Kripke frames and in particular is not Kripke-complete. There are three ways around this problem. The first is to restrict to fragments of $\glp$, such as the closed (variable-free) fragment of $\glp$. The second is to consider topological models rather than relational models. The third is to consider relational $\glp$-models, rather than $\glp$-frames. These last two approaches can further be combined, by considering topological $\glp$-models based on spaces which need not validate $\glp$, and this approach is what motivates the definition of the logic $\mathsf{J}$.

 \begin{definition}\label{def:J}
    The logic $\mathsf{J}$ is the $\lang$-logic whose rules are modus ponens and necessitation and whose axioms are:
  \begin{enumerate}[label=(\roman*)]
        \item All Boolean tautologies
        \item $[n](\phi\to\psi)\to([n]\phi\to[n]\psi)$;\label{ax:j-2}
        \item $[n]([n]\phi\to\phi)\to[n]\phi$ for all $n<\omega$;\label{ax:j-3}
        \item $[m]\phi\to[n][m]\phi$ for all $m\le n<\omega$;\label{ax:j-4}
        \item $\tup{m}\phi\to[n]\tup{m}\phi$ for all $m<n<\omega$;\label{ax:j-5}
        \item $[m]\phi\to[m][n]\phi$ for $m\le n < \omega$;\label{ax:j-6}
      \end{enumerate}
    If we substitute \ref{ax:j-6} with $[m]\phi\to[n]\phi$ for all $m\le n < \omega$, then we get the logic $\glp$. Note that $[m]\phi\to[n]\phi$ entails \ref{ax:j-6}, so $\mathsf{J} \subset \glp$.
\end{definition}
%In other words, it is the result of adding the schema $[m]\phi\to[m][n]\phi$ ($m\le n$) to $\Ig$ (cf. Definition~\ref{def:logic-I}). 
%We denote by $J_n$ the result of defining the logic $J$ with the language restricted to the modalities $\{[i]: i <n\}$ only. The logic $\glp$ is obtained from $J$ by adding the monotonicity axioms $[n] \varphi \to [m]\varphi$, for $n<m$.
$\mathsf{J}$ was introduced by Beklemishev \cite{Bek}, who proved that the logic is Kripke-complete, with respect to a type of Kripke model called a \textit{$\mathsf{J}$-tree}. The purpose of this article is to study the strong completeness of $\mathsf{J}$.

\begin{definition}
A logic $\mathsf{L}$ is \textit{strongly complete} with respect to a class of models $\mathcal{C}$ if every consistent set of $\mathsf{L}$-formul\ae\  has a model in $\mathcal{C}$.
\end{definition}
Here recall that a set of formul\ae\  $\Gamma$ is consistent relative to a logic $\mathsf L$ if $\mathsf L \not\vdash \neg \bigwedge \Delta$ for any finite $\Delta\subset\Gamma$.

It is well known that $\mathsf{GL}$ is not strongly complete with respect to its Kripke semantics. 
It follows from this that $\mathsf{J}$ is not strongly complete with respect to $\mathsf{J}$-trees. However, we show that it is strongly complete with respect to  \textit{$\mathsf{J}$-bouquets} -- slight variants of $\mathsf{J}$-trees with a modified satisfaction relation at root-like points. We prove:

\begin{theorem}\label{TheoremJIntro}
$\mathsf{J}$ is strongly complete with respect to the class of $\mathsf{J}$-bouquets.
\end{theorem}
Theorem \ref{TheoremJIntro} is a natural polymodal extension of the unimodal result from \cite{AG17}
and yields the following completeness theorem for $\glp$: if $\Gamma$ is a consistent set of formul\ae, then $\Gamma$ holds at a point of a $\mathsf{J}$-bouquet where all axioms of $\glp$ hold. 

Formally, $\mathsf{J}$-bouquets can be defined as a particular type of $\mathsf{J}$-spaces. These are defined as the spaces in which $\mathsf{J}$ is valid, although we shall give a combinatorial characterization of these of independent interest.
\begin{corollary}\label{CorollaryJIntro}
$\mathsf{J}$ is strongly complete with respect to the class of $\mathsf{J}$-spaces.
\end{corollary}

A long history of topological models for provability logic started with Esakia \cite{esakia}, who noticed the resemblance between the behaviour of the modal operator $\Diamond$ with that of the derivative operator in topologies and established completeness of $\gl$ for scattered spaces. The topological approach is especially fruitful for the polymodal generalizations of $\gl$, now extended by Corollary \ref{CorollaryJIntro} implies that $\glp$ is strongly complete for its general topological semantics (see Corollary \ref{CorollaryGLPStrongComp}).

\textit{Icard spaces} were studied by  Icard \cite{Ic08,Ic11} as a variant of the well known frame $\mathfrak{I}$ of Ignatiev \cite{IG} and provide a simple semantics for the closed fragment of $\glp$. 
Ignatiev's frame was later generalized into the blow-up constructions of Beklemishev \cite{Bek} and 
Beklemishev and Gabelaia \cite{BG13} developed the notion of a \textit{Beklemishev-Gabelaia-space} (or \textit{BG-space}) and proved that $\glp$ is complete with respect to these. Fern\'andez-Duque \cite{FD14} later extended this fact to the extension of $\glp$ equipped with transfinite modalities, as well as to general topological frames based on Icard spaces. Beklemishev and Wang \cite{BW24} have recently exhibited a simple class of \textit{periodic} topological frames which suffices for the completeness of $\glp$. 
The authors showed in \cite{AS24} that the closed fragment of $\glp$ is strongly complete with respect to a slight extension of the Ignatiev frame $\mathfrak{I}$ which includes its $\epsilon_0$th level, and exhibited counterexamples showing that both extension and the inclusion of points outside the main axis are necessary. However, Beklemishev has since pointed out that this result had been obtained earlier by Icard, and can be deduced from the results in Icard \cite[\S 3.1]{Icard2009}. 

The first author \cite{Ag22} had previously observed that $\glp$ is also not strongly complete with respect to its so-called ``standard topological models'' (see Beklemishev and Gabelaia \cite{BG14} and Bagaria \cite{Ba19} for more on these) and so at the current stage it is not clear whether there are any natural candidates for strongly complete (non-general) models. Nonetheless, Shamkanov \cite{Sh} has proved a nice \textit{global completeness} result making use of an illfounded proof system for $\glp$ and the global consequence relation, from which strong completeness for topological models follows, however it is open whether strong completeness holds for any tangible collection of ordinal or otherwise easily described spaces, and indeed this fails for what might otherwise be the first natural candidates, as shown in \S\ref{sec:conclusion}.

\subsection*{Acknowledgements}
This work was partially supported by the Austrian Science Foundation (FWF) through grants 10.55776/STA139, 10.55776/ESP3, and 10.55776/PAT2264325. The authors would like to acknowledge support from the Erwin Schr\"odinger Institute in Vienna during the thematic program \textit{Reverse Mathematics} in 2025, as well as to BIRS for support during the workshop \textit{Infinitary Proof Theory: Techniques and Applications} in 2025.

\section{Preliminaries}
We begin with some preliminary notions, definitions, and recall some relevant results. For general background on modal logic, we refer the reader to Blackburn, de Rijke, and Venema \cite{blackburn2001modal}.
For general background on provability logic, we refer the reader to Boolos \cite{Boolos}. 

We consider modal logic with infinitely many modalities $\{[n]: n\in\mathbb{N}\}$ as before. We write $\langle n \rangle = \lnot [n] \lnot$ for the dual operators.  In particular, we work with the logic $\mathsf{J}$ from Definition \ref{def:J}.
We denote by $\mathsf{J}_n$ the result of defining the logic $\mathsf{J}$ with the language restricted to the modalities $\{[i]: i <n\}$ only.

\begin{definition}\label{def:kripke}
    A {\it Kripke frame} is a tuple $F=(W,R_0, R_1,\dots)$, where $W$ is a set and 
    $R_i\subset W\times W$ for each $i<\omega$. Given a Kripke frame $F$ and a function $v:\vars\to \mathcal{P}(W)$,
    we say that $M=(F,v)$ is a {\it Kripke model}, which yields the following interpretation $\inter{\cdot}$ of modal formul\ae:
    \begin{itemize}
        \item $\inter{\bot} = \emptyset$;
        \item $\inter{p}=v(p)$, where $p\in\vars$;
        \item $\inter{\phi\land\psi} = \inter{\phi}\cap\inter{\psi}$;
        \item $\inter{\neg\phi} = W\setminus\inter{\phi}$;
        \item $\inter{\tup n\phi} = \set{x:\exists y\in\inter{\phi}\, xR_ny}$;
    \end{itemize}
We say that a formula $\phi$ holds at a point $x$ in a model $M$ if $x \in \inter{\phi}$, in which case we write $M,x \Vdash \phi$. We write $M \Vdash \phi$ to mean $M,x\Vdash \phi$ for some $x\in M$ and $M\models\phi$ to mean $M,x\Vdash\phi$ for all $x\in F$ and $F\models\phi$ to mean $M \models\phi$ for all models of the form $M = \langle F, v\rangle$ we might occasionally write $F,x\Vdash_v\phi$ instead of $(F,v),x\Vdash\phi$ or even omit the index, if $v$ is clear form the context.

\end{definition}

\subsection{Topological semantics}\label{subsec:top-sem}

We start with the definition of topological models for modal logic.

\begin{definition}
    Given a topological space $(X,\tau)$ for each $A\subset X$ we denote $d_\tau A = \set{x : \forall U\in\tau\exists y\ne x (y\in U\cap A)}$. We call $d_\tau$ the derivative operator. We omit the index if there's no risk of confusion.
\end{definition}

Topological spaces can be used to provide semantics for modal logics. In the polymodal case, we use \textit{polytopological spaces}, i.e., structures of the form $(X, \tau_i)_{i\in I}$ such that $(X,\tau_i)$ is a topological space for each $i$. The index set will always be either $\mathbb{N}$ or of the form $\{0,1,\hdots, n\}$ for some $n\in\mathbb{N}$.
\begin{definition}
    A \emph{topological model} for (poly)modal logic is a tuple $(X,\tau_i,v)_{i \in I}$, where $(X,\tau_i)_{i\in I}$ is a polytopological space and $v:\vars\to \mathcal{P}(X)$ is an interpretation. As before, $v$ is readily extended to arbitrary formul\ae:
    \begin{itemize}
        \item $ \inter{p} = v(p);$
        \item $\inter{\neg \phi} = X\setminus\inter{\phi}$; 
        \item $\inter{\phi\land\psi} = \inter{\phi}\cap\inter{\psi}$;
        \item $\inter{\tup{i}\phi} = d_{\tau_i}\inter{\phi}$;
    \end{itemize}
\end{definition}

Observe that Kripke frames can be regarded as topological spaces in which open sets are generated by \emph{cones} (i.e. upward closed sets), so the relational semantics is a particular case of topological semantics.

\begin{definition}
  Let $(X,\tau)$ be a scattered space. We define the \emph{rank function} $\rho_\tau:X\to\Ord$ by  $\rho_\tau(x) = \min\set{\alpha:x\notin d^{\alpha+1}X}$.
\end{definition}

Note that if $\tau$ is generated by the cones of a Kripke frame, then the notion of the topological rank coincide with the tree-rank.

\section{$\mathsf{J}$-spaces}
\begin{definition}
A polytopological space $(X, \vec \tau)$ is said to be a \textit{$\mathsf{J}$-space} if $\mathsf{J}$ is valid in $X$.
\end{definition}

The following result establishes a combinatorial characterization of $\mathsf{J}$-spaces.
\begin{theorem}\label{PropositionCharJSpaces}
Let $(X, \vec \tau)$ be a polytopological space.
 Then, $(X, \vec \tau)$ is a $\mathsf{J}$-space if and only if the following conditions hold for any $m < n < \omega$:
\begin{enumerate}
\item $\tau_m$ is scattered,\label{eq:j-space-1}
\item $d_{\tau_m} A \in \tau_{n}$ for all $A\subset X$, and \label{eq:j-space-2}
\item \label{eq:j-space-3} For all $A \in \tau_m$, there are disjoint $B_0, B_1 \in \tau_{n}$ such that the following hold: 
\begin{enumerate}
\item $A \subset B_0 \cup B_1$, \label{eq:j-space-3a}
\item $B_0 \subset A$,\label{eq:j-space-3b}
\item $B_0 \cup \{x\} \in \tau_m$ for all $x \in B_1$.\label{eq:j-space-3c}
\end{enumerate}
\end{enumerate}
\end{theorem}
Before proving the theorem, we show an additional implication of condition~\eqref{eq:j-space-3}.
\begin{claim}\label{claim:b-0-t-m}
  Let $X$ be a $\mathsf{J}$-space and let $A\in\tau_m$. Suppose $B_0, B_1\in \tau_n$ satisfy condition~\eqref{eq:j-space-3}. Then there is $B'_0\in\tau_n$ such that $B_0',B_1$ satisfy condition~\eqref{eq:j-space-3}, and $B_0'\in\tau_m$.
\end{claim}

\begin{proof}
    The proof is by case distinction:
    \begin{itemize}
        \item If $|B_1| > 1$, then there are two distinct $x_0,x_1\in B_1$, hence $B_0'=B_0 = (\set{x_0}\cup B_0)\cap(\set{x_1}\cup B_0)\in\tau_m$;
        \item If $B_1 = \emptyset$, then $B_0'=B_0 = A \in \tau_m$;
        \item If $B_1 = \set{x}$, then either $A = B_0$, and so $B_0$ is already open in $\tau_m$, or else $A = B_0 \cup B_1 = B_0\cup\set{x}$, in which case we may define $B'_0 = A$ and $B'_1 = \varnothing$, ensuring that $B'_0 \in \tau_m$.
    \end{itemize}
This completes the proof of the claim.
\end{proof}

The claim allows us to assume without loss of generality that $B_0\in\tau_m$ as well.

\begin{proof}[Proof of Theorem \ref{PropositionCharJSpaces}]

It is known \cite{Bl90} that a space is a model of $\gl$ if and only if it is scattered. Thus, we focus on the second and the third clause of the definition and Axioms \ref{ax:j-4}-\ref{ax:j-6}.

We start with showing that the topological conditions imply validity of the axioms.  Assume conditions~\eqref{eq:j-space-2} and \eqref{eq:j-space-3} hold.

To verify Axiom~\ref{ax:j-5}, we observe that by \eqref{eq:j-space-2} for a given $m<n$ we have $d_{\tau_m} A \in \tau_{n}$ for all $A\subset X$, then for any $\phi$, $\inter{\tup m \phi} = d_m\inter{\phi}\in \tau_{n}$, $\inter{\tup m \phi}\subset \inter{[n]\tup m \phi}$, and so $X\models\tup{m}\phi\to[n]\tup{m}\phi$.

To verify Axioms~\ref{ax:j-4} and \ref{ax:j-6}, we assume there is $x\in X$ such that $x\Vdash [m]\phi$, that is there is $U\in\tau_m$ with $x\in U$ such that $U\setminus\set{x} \subset\inter{\phi}$. Let
\[
  A= U\cap\set{y:\rho_{\tau_m}(y)<\rho_{\tau_m}(x)} \in \tau_m.
\]

\begin{claim}
  If $(X,\sigma)$ is a scattered topological space, then for any $y\in X$, the sets $Y=\set{x : \rho (x) < \rho(y)}$ and $Y\cup\set{y}$ are open.
\end{claim}
\begin{proof}[Proof of the Claim]
  Fix an $\alpha$. $\rho(x) \ge \alpha$ if and only if whenever $U\in\sigma$ and $x\in U$, then for each $\beta<\alpha$ there is a point $y\in U$ with $\rho(y)=\gamma$. It implies that for each $z\in Y$, there is $U\in\sigma$ with $U\subset Y$ and there is $V\in\sigma$ such that $y\in V$ and $V\subset Y\cap\set{y}$. So $Y$ and $Y\cup\set{y}$ are open.
\end{proof}

Trivially, $ A \subset U\setminus\set{x}$, so $A\subset\inter{[m]\phi}$.
Since $X$ is a $\mathsf{J}$-space, there are disjoint $V_0,V_1\in\tau_{n}$ with $\set{x}\cup A\subset V_0\cup V_1$ which satisfy property \eqref{eq:j-space-3} by Claim~\ref{claim:b-0-t-m} we can assume that $V_0\in \tau_m$. If $x\in V_0$, then $x\Vdash [n][m]\phi$ readily. Otherwise, by \eqref{eq:j-space-3c} and $V_0\subset\inter{\phi}$ we get $V_1\subset\inter{[m]\phi}$, it follows $x\Vdash [n][m]\phi$, and so $X\models [m]\phi\to[n][m]\phi$.

 On the other hand $V_0\subset \inter{\phi}$ hence $V_0\subset \inter{[n]\phi}$ hence $V_1\subset \inter{[m][n]\phi}$ and so $x\Vdash [m][n]\phi$ if $x\in V_1$. Moreover, $V_0\subset \inter{[n]\phi}$ implies $V_0\subset \inter{[m][n]\phi}$ since $V_0\in\tau_m$ as well, and so $X\models [m]\phi\to[m][n]\phi$.

 Thus, every space satisfying the conditions validates $\mathsf{J}$.

\medskip

For the opposite direction we assume that $X$ is a space that validates $\mathsf{J}$. For $A\in\tau_m$ and $p\in \vars$ let $v(p)=A$, then $d_{\tau_m}A = \inter{\tup mp}$. Since Axiom~\ref{ax:j-5} holds, for each $x\in d_{\tau_m}A $ we have $x\Vdash [n]\tup{m}p$. Topologically, this means that is there is $U_x\in \tau_n$ with $x\in U_x\in\tau_n$ such that $U_x\setminus \set{x}\subset d_{\tau_m}A$, then $d_{\tau_m}A = \bigcup_{x\in d_{\tau_m}A} U_x \in \tau_n$. Thus, \eqref{eq:j-space-2} holds.

Now for $A\subset X$ and $p,q\in \vars$ we let $B_0 = \bigcup\set{U\in\tau_n : U\subset A}$, $B_1 = \set{x \in X\setminus B_0 : \set{x}\cup B_0\in\tau_m}$ and $v(p)=A$, $v(q)=B_0$. We show that $B_0, B_1$ satisfy \eqref{eq:j-space-3} for $A$. Conditions \eqref{eq:j-space-3b}, \eqref{eq:j-space-3c} follow immediately as well as the fact that $B_0\in\tau_n$. It is now left to show that $B_1\in\tau_n$ and that $A\subset B_0\cup B_1$. For further convenience we first show that $B_0\in\tau_m$ necessarily. For the sake of contradiction assume $B_0\notin \tau_m$. Let $x \in B_0$ and $x$ not an interior point according to $\tau_m$. First,  $x\Vdash [m]p$. Moreover, since $x$ is not interior in $B_0$ for any $U\in\tau_m$ with $x\in U$ we have $U\setminus B_0\ne\varnothing$. Take $y\in U\setminus B_0$, if $y\in A$ then $y \Vdash \tup{n}\neg p$ by the construction of $B_0$ if $y\notin A$ then $y\Vdash \neg p$. And so $x\Vdash \tup{m}(\tup{n}\neg p\lor \neg p)$. It is a contradiction, since $X$ is a $\mathsf{J}$-space, it validates Axiom~\ref{ax:j-6} in particular its instance $[m]p\to[m][n]p$.

To show \eqref{eq:j-space-3a} we assume there is $x\in A\setminus B_0$ such that $\set{x}\cup B_0\notin \tau_m$. Since $B_0\in \tau_m$ we can conclude that each $\tau_m$-neighbourhood $U$ of $x$ contains a point $y\in X\setminus B_0$. From the construction of $B_0$ we can also conclude that every $\tau_n$ neighbourhood  of $y$ contains a point $z\in X\setminus A$. That is $x\Vdash \tup{m}\tup{n}\neg p$. Since $x\Vdash [m]p$ this contradicts the assumption that $X$ validates Axiom~\ref{ax:j-6}, and so $B_0, B_1$ cover $A$ and condition~\eqref{eq:j-space-3a} is satisfied. Now assume $B_1\notin\tau_n$ witnessed by $x\in B_1$ such that every $\tau_n$-neighbourhood of $x$ contains a point $y\in X\setminus B_1$ that is $\set{y}\cup B_0 \notin \tau_m$, hence $y\Vdash \tup{m}\neg q$ and $x\Vdash \tup{n}\tup{m}\neg q$ and $x\Vdash [m]q$ which contradicts Axiom~\ref{ax:j-4}, and so $B_1\in\tau_n$. Hence, a space is a $\mathsf{J}$-space if and only if it meets Conditions~\eqref{eq:j-space-1}-\eqref{eq:j-space-3}. 
\end{proof}

Note that if $(X, \vec \tau)$ is a $\glp$-space, then, in the notation of Theorem \ref{PropositionCharJSpaces}, we can always set $B_0 = A$ and $B_1 = \varnothing$.  Thus, the combinatorial characterization of $\mathsf{J}$-spaces reflects the fact that $\mathsf{J}$ is a sublogic of $\glp$.

\section{$\mathsf{J}$-trees and $\mathsf{J}$-bouquets}
Beklemishev \cite{Bek} has carried out a thorough semantic study of the logic $\mathsf{J}$, in particular introducing the following class of models:
\begin{definition}\label{def:j-tree}
  Let $n\in\mathbb{N}$. A \emph{$\mathsf{J}_n$-frame} is a pair $(T, \undertilde <)$ where $T$ is a set and $\undertilde <$ is a tuple $\tup{<_i}_{i<n}$ such that the following hold:
\begin{enumerate}
        \item $(T,<_i)$ is a converse well-founded relation on $T$ for each $i < n$;
        \item $\forall x, y\left(x <_n y \Rightarrow \forall z\left(x <_m z \Leftrightarrow y <_m z\right)\right)$ for each $m<n$;\label{eq:j-tree-2}
        \item $\forall x, y\left(x <_m y \land y <_n z \Rightarrow x <_m z\right)$ for each $m \leq n$;\footnote{Note that this entails that all $<_n$ are transitive.}\label{eq:j-tree-3}
\end{enumerate}
\end{definition}
$\mathsf{J}_n$-frames (and relational structures in general) can be viewed as polytopological spaces $(T, \vec\tau)$ where basic $\tau_i$-open subsets are cones of the form $\{y \in T: x <_i y\}$
for some $x\in T$. Thus, topological semantics extend the usual Kripke semantics for modal logic.

If $(T, \undertilde <)$ is a $\mathsf{J}_n$-frame, we may consider the reflexive and symmetric closure $E_i$ of $\bigcup_{i\leq j<n} <_j$. $E_i$-equivalence classes are called \textit{$i$-planes}. Observe that, according to the definition of a $\mathsf{J}_n$-frame, each $<_i$ induces a converse well-founded, transitive relation on the set of all $E_{i+1}$-planes, where $E_{n}$ is the identity. A \textit{$\mathsf{J}_n$-tree} is a $\mathsf{J}_n$-frame such that $(T/E_{i+1}, <_i)$ is a tree for each $i<n$. We introduce some notation: suppose $(T, \undertilde <)$ is a $\mathsf{J}_n$-tree. Then, we denote by $R^T_i$ the $<_i$-root of $T/E_{i+1}$ and by $r^T$ the hereditary root of $T$, that is $x \not <_i r^T$ for any $i<n$ and $x\in T$. Note that $R^T_{-1}=T$ is well defined.

\begin{theorem}[Beklemishev \cite{Bek}] \label{TheoremBekJTrees}
$\mathsf{J}_n$ is sound and complete with respect to the class of finite $\mathsf{J}_n$-trees.
\end{theorem}

It is well known that the provability logic $\mathsf{GL}$ is not strongly Kripke complete. Thus, strong Kripke completeness fails for $\mathsf{J}$ as well.
In this section we establish strong completeness of $\mathsf{J}$ with respect to a kind of Kripke-like structures inspired by the ``$\omega$-bouquets'' of \cite{AG17} and which we call \textit{$\mathsf{J}$-bouquets}.

\begin{definition}\label{def:bouquet}
  For $n\leq \omega$, a \emph{$\mathsf{J}_n$-bouquet} is a pair $(B, \undertilde <)$ where $B$ is a set, $\undertilde <$ is a tuple $\tup{<_i}_{i<n}$ and for each $m<n$, $(B, <_0, \hdots, <_{m-1})$ is a $\mathsf{J}_m$-tree in which every point has at most countably many $<_i$-successors and each point of successor $<_i$-rank has finitely many immediate $<_i$-successors for each $i$.

Given a $\mathsf{J}_n$-bouquet $(B, \undertilde <)$, we define a polytopology $\vec\tau$ based on $B$. The basic $\tau_i$-neighbourhoods of a point $x$ are defined by induction on the $<_i$-rank of $x$. We say $U\subset B$ is a basic $\tau_i$-neighbourhood of $x$ if $x \in U$ and one of the following holds:
\begin{enumerate}
  \item $x$ has finitely many immediate $<_i$-successors $x_0,\hdots, x_k$ and $U = \{x\} \cup \bigcup \{U_l: l\leq k\}$, where $U_l$ is a $\tau_i$-neighbourhood of $x_l$ for each $0\leq l \leq k$; or
  \item $x$ has infinitely many immediate $<_i$-successors $\{x_l:l\in\mathbb{N}\}$, and there is $l^*\in\mathbb{N}$ such that $U = \{x\} \cup \bigcup \{U_l: l^* \leq l\}$, where $U_l$ is a $\tau_i$-neighbourhood of $x_l$ for each $l^*\leq l$.

\end{enumerate}
  
We conventionally identify the $\mathsf{J}_n$-bouquet $(B,\undertilde<)$ with the corresponding polytopological space $(B,\vec\tau)$ and we say that a topological space $(B,\vec\tau)$ is a \emph{$\mathsf{J}$-bouquet} if it is a $\mathsf{J}_\omega$-bouquet.
\end{definition}

$\mathsf{J}$-bouquets are obtained as ``limit'' cases of $\mathsf{J}$-trees, with the only difference occurring at infinitely branching points. We shall retain the notation introduced for $\mathsf{J}$-trees, including the notion of an $i$-plane and the use of $R_i^B$ to denote the $<_i$-root of $B/E_{i+1}$. Note that in the original definition of bouquets in \cite{AG17}, new neighbourhoods were only added to the points of limit rank, in our case we do this at any point with infinitely many immediate successors, this only entails technical differences.

 %   moreover, there is a point $r\in B$ such that for no $b\in B$ and $i<\omega$, $b<_i r$ and for each $b\in B\setminus{r}$ there is $i<\omega$ such that $r<_i b$. We also define root planes $R_i$ for each $i$, namely $R_i = \set{b\in B : \forall b'\in B(b'\not<_ib)}$ note that $R_i = \set{r}\cup\bigcup_{j<i}{{<_j}(r)}$, where by ${<_j}(r)$ we conventionally denote $\set{b\in B : r <_i b}$.
    
%    For each $i$ and each $b\in R_i$, $b$ can have up to infinitely many $<_i$-successors, whereas each $b\in B\setminus{R_i}$ has finitely many $<_i$-successors.

\begin{lemma}\label{LemmaJBouquetsAreJSpaces}
Suppose $(B, \undertilde <)$ is a $\mathsf{J}$-bouquet. Then, $(B, \undertilde <)$  is a $\mathsf{J}$-space.
\end{lemma}
\proof
We prove this using the characterization given by Theorem \ref{PropositionCharJSpaces}. Obviously, each $<_i$ generates a scattered space due to well-foundedness. Now, in order to show \eqref{eq:j-space-2}, we take $A\subset B$ and consider $d_m A$ for some $m$. Recall that,
\[
  d_m A = \set{x\in B : U\in\tau_m \land x\in U \implies \exists y\ne x, y\in A\cap U }.
\]

Note that Definition~\ref{def:bouquet}, if $U\in \tau_m$, $x\in U$ and $x<_n y$ for $n>m$, then $\set y \cup (U\setminus \set{x})\in \tau_m$, since ${<_m}(x)={<_m}(y)$ by Definition~\ref{def:bouquet}. It follows that $d_m A$ is a union of $<_n$-cones, which are obviously $\tau_n$-open. 

To show \eqref{eq:j-space-3} we take an $m$-open set $U$. Since, $U$ is an $m$-cone, we can take a set $M$ of $<_m$-minimal elements of $U$, then $U = \bigcup_{x\in M} A_x$ where either $A_x = \set{x}\cup\bigcup\set{U_l : l<_m k}$ or $A_x = \set{x}\cup\bigcup\set{U_l : l^* <_m l}$ (as in Definition~\ref{def:bouquet}). One can see then that in both cases $B_1 = \set{y : \exists x\in M (y=x \lor x<_n y)}$ and $B_0 = \set{y : \exists x\in M(x<_m y)}$ are as required.
\endproof

For our completeness proof, it, in fact, suffices to consider $\mathsf{J}$-bouquets which are \textit{small}, in the following sense:
\begin{definition}\label{def:small-bouquet} 
  A $\mathsf{J}$-bouquet $(B, \undertilde <)$ is \textit{small} if for all $i\in\mathbb{N}$, $(B, \undertilde <)$ has at most one infinitely branching $(i+1)$-plane, in which case it is equal to $R^B_i$.
\end{definition}

\begin{lemma}\label{LemmaJBouquetChar}
Suppose that $(B,\undertilde <)$ is a small $\mathsf{J}$-bouquet. Then, for each $i\in\mathbb{N}$, each formula $\langle i\rangle \phi$, and each $x \in B$, we have:
\begin{enumerate}
  \item If $x$ has finitely many $<_i$-successors, then $B, x\Vdash \langle i\rangle \phi$ if and only if $B, y\Vdash \phi\lor\tup{i}\phi$ for some immediate $<_i$-successor of $x$.
\item If $x$ has infinitely many $<_i$-successors, then $B, x\Vdash \langle i\rangle \phi$ if and only if $B, y\Vdash \phi\lor\tup{i}\phi$ for infinitely many immediate $<_i$-successors of $x$.
\end{enumerate}
\end{lemma}
\proof
By the definition of small bouquet $x\notin R_i^B$ implies that any neighbourhood of $x$ contains ${<_i}(x)\cup\set{x}$, from which the first clause follows. If $x\in R_i^B$ and $x$ has infinitely many successors, then $x\Vdash \tup{i}\phi$ if and only if for each neighbourhood $U$ of $x$ there is $y\in U\setminus\set{x}$ with $y\Vdash \phi$,if and only if for infinitely many $l$ there is $y_l\in U_l$ (we use notation from Definition~\ref{def:bouquet}) with $y_l\Vdash \phi$ and so $x_l\Vdash \phi\lor\tup{i}\phi$.
\endproof

This allows us freely argue about bouquets as if they were a modification of Kripke frames.

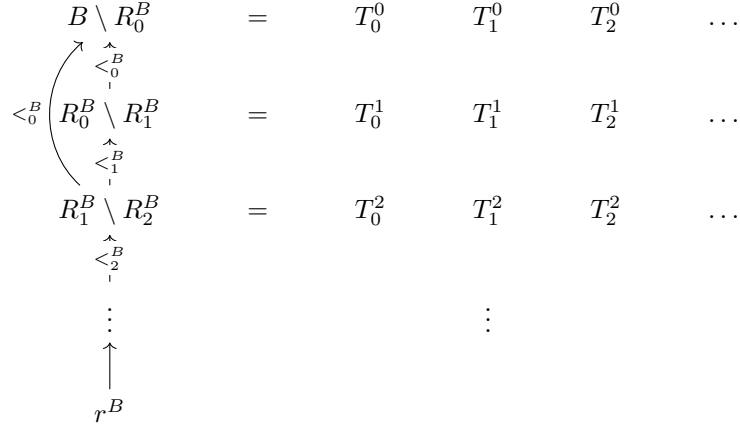
\begin{figure}
\begin{center}
% https://tikzcd.yichuanshen.de/#N4Igdg9gJgpgziAXAbVABwnAlgFyxMJZABgBoAWAXVJADcBDAGwFcYkQAnAPQCEQBfUuky58hFAGZSE6nSat2AHUW0oEHAkHDseAkTIAmWQxZtEIAEq8A+gEZlcGDgC2WMMzgACKz2sGBQiAYOmJEBqRGNCYK5gAqXAbWxAHaonqSEcbyZiDxibYpQSK64sjkmVHZ7Hl+hcFppQCsFXKmSopqGnXFoShktllt5j5JDk6u7l4WdrzdIenI4QOVQ7lctklzDURSy60xaxsFWkXzpeV70TnxG-4n9SVEzZdV5sqdmoEPvYukxIMHeLETb3HoLKT-FaArjA45fMHnP4A64w2qgs5PJFQnLvdSfVKPFC2FpXdgAXi2hOQxJeqwp6O2RKx+xy9PhGL6zNJ5h4Yxcbg83hsyQZVLIMmx7VUeIEshgUAA5vAiKAAGYcCDOJAAdhoOAgSFsAE5JeYADxJWY0Rj0ABGMEYAAUEew3NhYIV1ZqkOEQPqdda7Q7nRyQG6sB7TSALbZZicvVrEL7-YhjTR7WAoEhyCaWewLcQ44EE0gyH6DUniPGNYmDGWU77udG-FaQDb7U6XeZwx7+JR+EA
\begin{tikzcd}
B\setminus R^B_0                                                                      & = & T^0_0 & T^0_1  & T^0_2 & \dots \\
R^B_0\setminus R_1^B \arrow[u, "<_0^B" description]                                   & = & T^1_0 & T^1_1  & T^1_2 & \dots \\
R^B_1\setminus R^B_2 \arrow[u, "<_1^B" description] \arrow[uu, "<_0^B", bend left=49] & = & T^2_0 & T^2_1  & T^2_2 & \dots \\
\vdots \arrow[u, "<_2^B" description]                                                 &   &       & \vdots &       &       \\
r^B \arrow[u]                                                                         &   &       &        &       &      
\end{tikzcd}
\end{center}
\caption{A small $\mathsf{J}$-bouquet with root $r^B$.}
\end{figure}

\section{Strong completeness} 
In this section, we prove the main result of this article:

\begin{theorem}\label{thm:bouquet-n}
  Suppose that for some $n\le \omega$, $\Gamma$ is a set of $\lang_n$-formul\ae\  consistent with $\mathsf{J}$. Then there is a small $\mathsf{J}_n$-bouquet $(B, \undertilde <)$ and a valuation $v$ such that $B,r^B\Vdash_v \Gamma$.
\end{theorem}

To prove the theorem, let $\Gamma$ be a set of formul\ae\  consistent with $\mathsf{J}$. To simplify the proof we prove the case with $n=\omega$, one can trivially adapt the proof for a finite $n$.
By extending it if necessary, we may assume that $\Gamma$ is maximal consistent. Throughout the proof, we assume that all formulae are presented in the negation normal form.

The proof goes as follows. In Subsection~\ref{SubsectionGamma} we define finite sets $\Gamma(i)$ of formul\ae\ which approximate $\Gamma$ and we show their consistency. In Subsection~\ref{SubsectionBouquet}, using the consistency and finiteness of the previously defined sets of formul\ae\, we amalgamate the finite models of these sets $T_i$, and show that the resulting model $B$ is a $\mathsf{J}$-bouquet. Finally, in Subsection~\ref{SubsectionTruth}, we show that $B$ satisfies $\Gamma$.

\subsection{Slicing $\Gamma$}\label{SubsectionGamma} The first step of our proof is to define finite set of formul\ae\ $\Gamma(i)$ for each $i$ and show its consistency.

We let $N^\bot=\set{n<\omega: [n]\bot\in \Gamma}$ and $N^\top=\set{n<\omega:\tup n\top\in \Gamma}$, note that $N^\bot \sqcup N^\top=\omega$. Fix an enumeration $(\psi_i)_{i<\omega}$ of all formul\ae\  such that $\tup{n_i}\psi_i\in\Gamma$ in which
each $\psi_i$ occurs infinitely often. Using the same indexing $i\mapsto n_i$, we fix an enumeration $(\phi_i)_{i<\omega}$ of all formul\ae\  such that $[n_i]\phi_i\in\Gamma$ and $\langle n_i\rangle \top \in \Gamma$ (for $n$ such that $\langle n\rangle \top \not\in \Gamma$, $n$ does not belong to the range of the mapping $i\mapsto n_i$, and $[n]\psi$ never gets added to the enumeration). We start by defining sets of formul\ae\:
\begin{equation*}\label{eq:Gamma-i-prime}
\begin{aligned}
    \Gamma'(i) = \set{\psi_i}\cup &\Big\{\phi_j\wedge[n_j]\phi_j: j<i, n_j = n_i\Big\}\\
\cup& \Big\{[n_j]\phi_j\land \tup{n_j}\psi_j : j<i, n_j < n_i\Big\}
\end{aligned}
\end{equation*}
Each $\Gamma(i)$ is finite, so for each $i$ there is a finite $\mathsf{J}$-tree $S_i$ with root $s_i$ such that $S_i,s_i\Vdash\Gamma'(i)$. If we let $S = \set{r}\cup\bigcup_{i<\omega} S_i$, where $r$ is the root below all $S_i$, then $S,r\Vdash \Gamma$. However, the model $S$ need not be a model of $\mathsf{J}$ and an additional closure is required. To ensure that this closure does not violate the indented satisfaction, we need to extend $\Gamma'(i)$. For each $i$, we define the maximal index of a modal operator appearing in $\Gamma'(i)$. Setting $m_{-1}=0$, we let:
\begin{align*}\tag*{($\bullet$)}\label{eq:m_i}
  m_i=\max\set{m :\ & m=m_{i-1}\text{, or }\\&m=n_i \text{, or }\\& \exists \theta ([m]\theta\in\sub{\Gamma'(i)}\lor \tup m\theta\in\sub{\Gamma'(i)}},
\end{align*}
We define the following sets of formul\ae\  for each $i\leq \omega$:
  \begin{align*}\tag*{($\star$)}\label{eq:Gamma-i}
    \Gamma(i) = \Gamma'(i)\cup &\Big\{ [m]\phi_j \land [n_j][m]\phi_j : j<i, n_j \le n_i, n_j\le m\le m_i\Big\}\\
\cup& 
\Big\{[k]\bot  : k\in N^\bot\text{, and }k < m_i\Big\}.
\end{align*}
The goal now is to build a $\mathsf{J}$-bouquet by amalgamating models for each $\Gamma(i)$.  This amalgamation is not simply their union, however, as the definition of a $\mathsf{J}$-bouquet will impose non-trivial interactions between the models, the formulas in $\Gamma(i)\setminus\Gamma'(i)$ help us retain validity through these interactions. We begin with a claim:

%  where $(\psi_i)_{i<\omega}$ is an enumeration of all formul\ae, such that $\tup{n_i}\psi_i\in\Gamma$ and
%  each $\psi_i$ occur infinitely often, $(\phi_i)_{i<\omega}$ enumerates all formul\ae\  such that 
%  $[n_i]\phi_i\in\Gamma$.
%\end{definition}

%\begin{proof}[Outline of the proof of Theorem \ref{thm:bouquet-n}]
%    Given a $\Gamma \subset \mathcal L$ we define sets of formul\ae\  as 
%  follows: we take sets $\Gamma(i)$ (as defined in Def.~\ref{def:gamma-i}) for all $i<\omega$. Provided that $\Gamma(i)$ is $J$-consistent, each $\Gamma(i)$ has a model $(T_i,\mathbf R_i,v_i)$ with a root $r_i$.
%  Then we add a root $r$ and for each $T_i$ we add $\tup{r,r_i}$ to the
%  $<_{n_i}$ and close the relations under $J$ rules. The last step is to show that this model satisfies $\Gamma$ exactly in the point $r$.
%\end{proof}

%\begin{proof}[Proof of Theorem \ref{thm:bouquet-n}]

%    Following the outline we start with the following:

      \begin{claim}\label{ClaimJStrCompCons}
        $\Gamma(i)$ is $\mathsf{J}$-consistent for each $i<\omega$.
      \end{claim}
      
      \begin{proof}
It suffices to show that $\Gamma \vdash \tup {n_i}\bigwedge\Gamma(i)$. Towards a contradiction assume that $\Gamma \vdash [n_i]\bigvee \neg\Gamma(i)$.
First, using the axioms we obtain
$\Gamma \vdash [n_j]\phi_j, \Gamma \vdash [n_i][n_j]\phi_j$ for $j\le i$, $n_j\le n_i$. Next, $\Gamma\vdash [n_i]\tup{n_j}\psi_j$ for $j<i$ and $n_j<n_i$,  then also $\Gamma\vdash [n_i][k]\bot$ for $k< n_i$ with $k\in N^\bot$. Finally, $\Gamma\vdash [n_i]([m]\phi_j\land [n_j][m]\phi_j)$ for each $j\le i$, $n_j\le n_i$ and $ m\ge n_i$. Hence,
\begin{align*}
  \Gamma \vdash
  [n_i]\Big(&\bigwedge_{j<i,n_j=n_i}(\phi_j\land[n_j]\phi_j)\land
    \bigwedge_{j<i,n_j<n_i}(\tup{n_j}\psi_j\land[n_j]\phi_j) \land\\
&  \bigwedge_{k<n_i,k\in N^\bot}[k]\bot \land \bigwedge_{j<i, n_j\le n_i, n_j\le m\le m_i}[m]\phi_j \land [n_j][m]\phi_j \Big).\end{align*}
      Together with the assumption $\Gamma \vdash [n_i]\bigvee \neg\Gamma(i)$ we get $\Gamma\vdash[n_i]\neg\psi_i$, note that $\tup{n_i}\psi_i\in\Gamma$, a contradiction.
      \end{proof}

      \subsection{Constructing the bouquet}\label{SubsectionBouquet}

      \begin{claim}
        There are an $n<\omega$ and a finite $\mathsf{J}_n$-tree $\tup{T_i, \undertilde <^{T_i}}$ with an evaluation $v_i$ such that $T_i,r^{T_i}\Vdash \Gamma(i)$
        for each $i<\omega$. In fact, we can take $n=m_i$ (see \ref{eq:m_i}).
      \end{claim}
\proof
Immediate from Claim \ref{ClaimJStrCompCons} and Theorem \ref{TheoremBekJTrees}.
\endproof

Now we proceed to the construction of a $\mathsf{J}$-bouquet $B$  such that $B,r^B\Vdash \Gamma$. The construction is by stages. We begin with
 $B_{-1} := \set{r}$.
Inductively, $B_i$ is defined by putting $B_i = B_{i-1} \cup T_i$. We still need to define the relations $\undertilde <^{B_i}$. 
The idea for this is that we would like to place $T_i$ precisely $<_{n_i}$-above $r^B$. However, we must add more arrows in order to ensure that the result is a $\mathsf{J}$-tree. We do so according to the following case distinction
(see Figures \ref{FigureDefRelationsBouquet1}--\ref{FigureDefRelationsBouquet3}).
\begin{align*}\tag*{(\#)}\label{eq:rel-closure}
<^{B_i}_\ell = 
\begin{cases}
        <^{B_{i-1}}_\ell \cup <_\ell^{T_i} \cup\, \Big( R_{\ell}^{B_{i-1}} \times R_{\ell-1}^{T_i}\setminus R_\ell^{T_i} \Big)\\\quad\quad\cup\,\Big( R_\ell^{T_i} \times R_{\ell-1}^{B_{i-1}}\setminus R_\ell^{B_{i-1}} \Big),& \text{ if } \ell < n_i;\\
        <^{B_{i-1}}_\ell\cup <_\ell^{T_i} \cup\, \Big(R_{n_i}^{B_{i-1}} \times R_{n_i-1}^{T_i}\Big) ,& \text{ if } \ell = n_i;\\
        <^{B_{i-1}}_\ell\cup <_\ell^{T_i} ,& \text{ if } \ell > n_i;\\
\end{cases}
\end{align*}

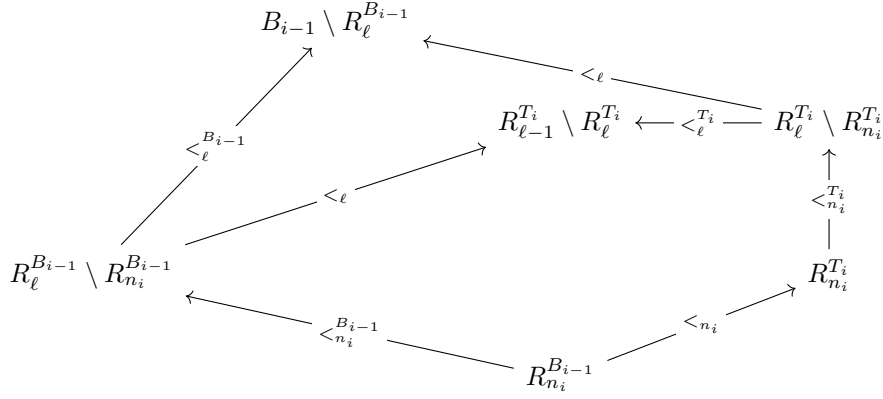
\begin{figure}
\begin{center}
% https://tikzcd.yichuanshen.de/#N4Igdg9gJgpgziAXAbVABwnAlgFyxMJZABgBoBmAXVJADcBDAGwFcYkQAlAfWACsBfAHrAAQ\ellywBaAIz9+AHTlwYOALZYwzOAAJuwMFyxDR46bJD9S6TLnyEUAFlJTqdJq3bdewgCoH5i5TUNbV19Qx8-c0sQDGw8AiIpUmIXBhY2RBAxYEkZBSVVdU0dLi9\ellHNN+KKs42yIAJidUtwzOHl5TCMN8wKKQ0q6qixqbBJRG+2b0\ellx4wo2zcs2GY63i7ZEcqG\ellT3TNC-QfMXGCgAc3giUAAzACcIFSRHEBwIJDIQRnoAIxhGAAVVnVMupsLAQNsWuwADyzA7lRZDaK3e6PGgvJAAVgh00yML0kRonx+-0BYxAIKwYOWyIeiCxz1eiCSrhxIDxc0OhO+vwBtTJFKpSLutOZ6MQ5GxuzZA2AvkM4I+3JJfLs5LAoLY1OFbzR\ellMaLKlMLKC0qCqJPNJqoFmqFKMQ7zFEsVxN5oyt6spbBoPzAUB1BtaRuqIBpSFFeu9MF9-vNyrd7GtCp2gdKR34QA
\begin{tikzcd}
                                                                                                                     & B_{i-1}\setminus R_\ell^{B_{i-1}} &                                                                                                   &  &                                                                                                      \\
                                                                                                                     &                                & R_{\ell-1}^{T_i}\setminus R_\ell^{T_i}                                                                  &  & R_\ell^{T_i}\setminus R_{n_i}^{T_i} \arrow[ll, "<_\ell^{T_i}" description] \arrow[lllu, "<_\ell" description] \\
                                                                                                                     &                                &                                                                                                   &  &                                                                                                      \\
R_{\ell}^{B_{i-1}}\setminus R_{n_i}^{B_{i-1}} \arrow[ruuu, "<_\ell^{B_{i-1}}" description] \arrow[rruu, "<_\ell" description] &                                &                                                                                                   &  & R_{n_i}^{T_i} \arrow[uu, "<_{n_i}^{T_i}" description]                                                \\
                                                                                                                     &                                & R_{n_i}^{B_{i-1}} \arrow[llu, "<_{n_i}^{B_{i-1}}" description] \arrow[rru, "<_{n_i}" description] &  &                                                                                                     
\end{tikzcd}
\end{center}
\caption{Definition of $<^{B_i}_\ell$ when $\ell < n_i$.}\label{FigureDefRelationsBouquet1}
\end{figure}
\begin{figure}
\begin{center}
% https://tikzcd.yichuanshen.de/#N4Igdg9gJgpgziAXAbVABwnAlgFyxMJZARgBoAWAXVJADcBDAGwFcYkQAlAfWACsBfAHrAAQ\ellywBaYv34h+pdJlz5CKAEyk11Ok1btufIcAAqXLLPmLseAkQ0AGbQxZtEnHrylHT5gDq+4GBwAWywwZ\ellgAAgMBYR8LBRAMaxUie00nXVcQMWBJaX9AkLCI6I8\ellXPyZOW0YKABzeCJQADMAJwhgpHSQHAgkAGYaZz03AB4uXmFKr1kaRnoAIxhGAAUlG1UQMOxYOUT2zu6aPqQyEAXltY3Utx2sPeGs9gneEBoACxh6KCRwxkYlhAhy6iHOp0QGguYWycAg\ellAe7xAXx+f2YAJO9Cw\ellHYkDAbHmSxW6xSt\elluYF2BJ0LhekziZgslH4QA
\begin{tikzcd}
                                 &                                                                                         & R_{\ell-1}^{T_i}\setminus R_{\ell}^{T_i}              \\
                                 &                                                                                         &                                                 \\
B_{i-1}\setminus R_{\ell}^{B_{i-1}} &                                                                                         & R_{\ell}^{T_i} \arrow[uu, "<_\ell^{T_i}" description] \\
                                 &                                                                                         &                                                 \\
                                 & R_{\ell}^{B_{i-1}} \arrow[luu, "<_\ell^{B_{i-1}}" description] \arrow[ruu, "<_\ell" description] &                                                
\end{tikzcd}
\end{center}
\caption{Definition of $<^{B_i}_\ell$ when $\ell = n_i$.}\label{FigureDefRelationsBouquet2}
\end{figure}
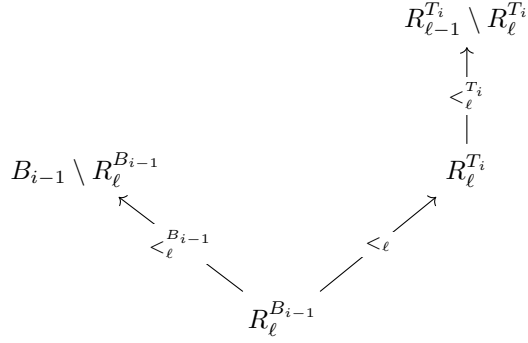
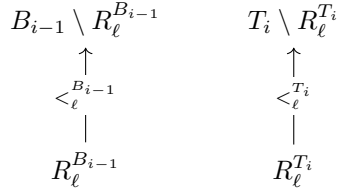
\begin{figure}
\begin{center}
% https://tikzcd.yichuanshen.de/#N4Igdg9gJgpgziAXAbVABwnAlgFyxMJZABgBpiBdUkANwEMAbAVxiRACEB9YLAWgEYAvgB1hcGDgC2WMEzgACAEqcAVgD1gXHgMGCQg0uky58hFP3JVa\ellFmwAqnLKPFSZcpao0Oseg0ex4BERkAExW9MysiCDK6prcfEK+hiAYAaZEFmHUEbbRsV6OvlYwUADm8ESgAGYAThCSSCHUOBBIZNaRbAA8nvHaSSDUDHQAR\ellAMAArGgWYgMtiw+il1DUgAzC1tiBadeSC9cd56w2MT0+lB0QtYS4IUgkA
\begin{tikzcd}
B_{i-1}\setminus R_\ell^{B_{i-1}}                        & T_i\setminus R_\ell^{T_i}                        \\
                                                      &                                               \\
R_\ell^{B_{i-1}} \arrow[uu, "<_\ell^{B_{i-1}}" description] & R_\ell^{T_i} \arrow[uu, "<_\ell^{T_i}" description]
\end{tikzcd}
\end{center}
\caption{Definition of  $<^{B_i}_\ell$  when $n_i < \ell$.}\label{FigureDefRelationsBouquet3}
\end{figure}

      Then we put $B=\bigcup_i B_i$ and $<_{\ell}^B = \bigcup_i <_{\ell}^{B_i}$ for each $\ell<\omega$ (in particular one can see that $R_\ell = \bigcup_i R_{\ell}^{B_i}$). Finally, we define a valuation on $B$ by 
\[v(p) = \{r : p\in \Gamma\} \cup\bigcup_i v_i(p).\]
It remains to verify that this model is as desired. In the coming proofs, if $x<_\ell^B y$ for some $\ell$, it is convenient to keep track of the stage where $(x,y)$ was added. At each stage we can split $<_\ell^{B_i} = <_\ell^{B_{i-1}} \cup <_\ell^{T_i} \cup\  Q_\ell^i$, where
\begin{align*}\tag*{($\dagger$)}\label{EqQij}
Q^{i}_\ell =
\begin{cases}
        \Big( R_{\ell}^{B_{i-1}} \times R_{\ell-1}^{T_i}\setminus R_\ell^{T_i} \Big)\, \cup\,\Big( R_\ell^{T_i} \times R_{\ell-1}^{B_{i-1}}\setminus R_\ell^{B_{i-1}} \Big),& \text{ if } \ell < n_i;\\
  R_{n_i}^{B_{i-1}} \times R_{n_i-1}^{T_i} ,& \text{ if } \ell = n_i;\\
        \emptyset ,& \text{ if } \ell > n_i;\\
\end{cases}
\end{align*}
It is useful to see that if $(x,y)\in <_\ell^{B}$ for some $\ell$, then there is a unique $i$ such that either $(x,y)\in <_\ell^{T_i}$ or $(x,y)\in Q_\ell^i$.
      
      \begin{lemma}\label{ClaimBIsaJbouquet}
        $(B, \undertilde <)$ is a $\mathsf{J}$-bouquet, moreover for every $i$, $B_i$ is a $\mathsf{J}$-tree.
      \end{lemma}
      \begin{proof}
        We first show that $<_\ell^B$ is wellfounded for each $\ell$. Assume for some $\ell$ there is an infinite $<_\ell^B$-increasing sequence of points $x_0 <_\ell^B \dots <_\ell^B x_k <_\ell^B\dots$. We have $x_1\in T_n$ for some $n$, since $x_0 <_\ell x_1$ we have $x_1\in B_m \setminus R_\ell^{B_m}$ for large enough $m$, therefore it can only be the case that $x_k\in T_n$ for all $k>0$. Since $T_n$ is finite, there is no such sequence.

        Now that we showed well-foundedness, in the rest of the proof we focus on showing that each $B_i$ and $B$ are $\mathsf{J}$-trees by verifying properties \eqref{eq:j-tree-2}, \eqref{eq:j-tree-3} from Definition~\ref{def:j-tree}. Fix $j<\omega$.

        To verify property~\eqref{eq:j-tree-3}  we assume that $x,y,z\in B_j$ are such that $x <_\ell^B y <^B_k z$ where $\ell \le k$, and show that $x<_\ell z$. Let $i<j$ be such that $(x,y)\in {<_\ell^{T_i}}$ or $(x,y)\in Q_\ell^i$.
        \begin{itemize}
          \item if $(x,y)\in <^{T_i}_\ell$ then $y\notin R_k^{B_{i'}}$ for any $i'\ge i$, it follows that $x <_\ell^{T_i} y <_k^{T_i} z $ and we are done, since $T_i\models \mathsf{J}$.
            \item otherwise, if $n_i=\ell$, it must be $x\in R_\ell^{B_{i-1}}$ and $y\in R_{\ell-1}^{T_i}$ it follows that $y<_k^{T_i} z$, hence $z\in R_{\ell-1}^{T_i}$ and so $x<_\ell^{B_i} z$.
            \item if $\ell< n_i$ we have $y\in R_{\ell-1}^{B_{i-1}}\setminus R_{\ell}^{B_{i-1}}$ implies $z\in R_{\ell-1}^{B_{i'-1}}\setminus R_{\ell}^{B_{i'-1}}$ for some $i'\ge i$ as well as $y\in R_{\ell-1}^{T_{i-1}}\setminus R_{\ell}^{T_{i-1}}$ implies $z\in R_{\ell-1}^{T_{i-1}}\setminus R_{\ell}^{T_{i-1}}$. Both cases imply $x <_\ell^{B} z$.
          \end{itemize}
          This means that $B_j$ satisfies \eqref{eq:j-tree-3}. Since for each $x,y,z\in B$, there is $j$ such that $x,y,z\in B_j$, it follows that $B$ satisfies \eqref{eq:j-tree-3} as well.

          \medskip
        
          To verify property~\eqref{eq:j-tree-2}  assume that $x,y\in B_j$ are such that $x<_\ell y$ from some $\ell$. We need to show that whenever $z\in B_j$ and $k<\ell$, then $x<_k z$ if and only if $y<_k z$.
          Let $i$ be such that $(x,y)\in <_\ell^{T_i}\cup Q_\ell^i$, recall definition~\ref{EqQij} of $Q_\ell^i$. Note also that since $x<_\ell y$ and $\ell>k$, it follows that $x E_{k+1} y$ (see the paragraph after Definition~\ref{def:j-tree}), in particular $x \in R_k^S$ if and only if $y\in R_k^S$, where $S$ is some $\mathsf{J}$-tree.
          \\   {\large Case I.} $x<_\ell^{T_i} y$. Then $x,y\in T_i$ then either $z\in T_i$ and we are done since $T_i\models \mathsf{J}$, or else:
           \begin{itemize}
             \item if $z\in B_{i-1}$, then we have $x<^{B_i}_k z$ if and only if $y<^{B_i}_k z$ if and only $z \in R_{k-1}^{B_{i-1}}\setminus R_{k}^{B_{i-1}}$ and $x,y \in R^{T_{i}}_k$; 
             \item if $z \in B_{i'}\setminus B_{i'-1} = T_{i'}$ for some $i'>i$, then $x <_k^{B_{i'}} z$ if and only if $y<_k^{B_{i'}}z$ if and only if $z \in R_{k-1}^{T_{i'}}\setminus R_{k}^{T_{i'}}$ and $x,y \in R^{B_{i}}_\ell$;
            \end{itemize} \noindent 
            {\large Case II.} $(x,y)\in Q_\ell^i$. In this case, $x\in B_{i-1}$ if and only if $y\in T_i$ and $x,y\in R_k^{B_i}$.
            \begin{itemize}
                \item if $z \in B_{i-1}$ we have $x<_\ell^{B_i} z$ if and only if $y<_\ell^{B_i} z$ if and only if $z\in R_{\ell-1}^{B_{i-1}}\setminus R_{\ell}^{B_{i-1}}$;
                \item if $z \in T_{i'}$ for some $i'\ge i$, we have $x<_\ell^{B_{i'}} z$ if and only if $y<_\ell^{B_{i'}} z$ if and only if $z\in R_{k-1}^{T_{i'}}\setminus R_{k}^{T_{i'}}$.
        \end{itemize}
        Similarly for any $x,y,z\in B$, there is $j$ such that $x,y,z\in B_j$, so the fact that $B_j$ satisfies \eqref{eq:j-tree-2} implies that so does $B$. 
      \end{proof}
        
\begin{claim}\label{ClaimBIsaJbouquetSmall}
$B$ is a small $\mathsf{J}$-bouquet.
\end{claim}
\proof
By Claim \ref{ClaimBIsaJbouquet}, $B$ is a $\mathsf{J}$-bouquet. We are to show that for each $x\notin R^B_\ell$, $x$ has finitely many $<_\ell$ successors. Now, fix $\ell\in\mathbb{N}$ and let $x\not\in R_\ell^B$, say, $x \in T_{i+1}\setminus T_i$. Recalling the definition of $<^B_\ell$, points $y$ such that $x <^B_\ell y$ either  belong to $T_{i+1}$ or to $B_i$ (this uses that $x \not \in R_\ell^B$). Since both $T_{i+1}$ and $B_i$ are finite, the claim follows.
\endproof

\subsection{Validating $\Gamma$}\label{SubsectionTruth}

Before proceeding to show that $B$ satisfies $\Gamma$, we need the following claim. Recall that $N^\bot=\set{n<\omega: [n]\bot\in \Gamma}$.

\begin{claim}\label{claim:empty-root}
  For each $\ell \in N^\bot$, for each $x\in B$, $r\not<_\ell x$.
\end{claim}
\begin{proof}
  In fact, we need to show that for each $i<\omega$, for each $i<\omega$, $r\not <_\ell^{B_i} x$ for any $x\in B_i$. By construction and the induction hypothesis, this could fail in the following case: $x\in R_{\ell-1}^{T_i}\setminus R_\ell^{T_i}$ and $\ell<n_i$, which means that $r^{T_i}<_{\ell} x$. It follows that $T_i,r^{T_i}\Vdash \tup{\ell} \top$, but $\ell<n_i$ and so $[\ell]\bot\in \Gamma(i)$ (see \ref{eq:Gamma-i}), a contradiction. 
\end{proof}

By Claim \ref{ClaimBIsaJbouquetSmall}, $(B,\undertilde <)$ is a small $\mathsf{J}$-bouquet. We now prove the lemma with which the proof of the theorem will be complete.
\begin{lemma}[Truth Lemma]
  \[
    \label{eqJStrongCompInduction}
    \forall \phi\, (\phi \in \Gamma\leftrightarrow B, r^B \Vdash \phi),
  \]
\end{lemma}
\begin{proof}
The proof is by induction.
If $\phi$ is a propositional variable or a negation of a propositional variable, then this follows immediately from the definition of the valuation. The Boolean case is immediate. Thus, it suffices to prove by induction on $\phi$ that for all $k\in\mathbb{N}$, we have
\begin{equation}\label{eqJStrongCompInductionBox}
[k] \phi \in \Gamma\leftrightarrow B, r^B \Vdash [k] \phi.
\end{equation}

First, if $k\in N^\bot$, then by Claim~\ref{claim:empty-root} $B,r^B\Vdash [k]\bot$ and $[k]\bot\in \Gamma$.
Thus, for the rest of the proof we assume that $k\in N^\top$, so that $\langle k\rangle \top \in \Gamma$.
We shall use the characterization given by Lemma \ref{LemmaJBouquetChar} and \ref{eq:rel-closure}.

Let $i^*$ be large enough so that for all $l \leq k$ and all subformul\ae\  $\theta$ of $\phi$ and  $\neg\phi$\footnote{Recall that we assume all formul\ae\ are in negation normal form.}, if $[l]\theta \in \Gamma$, then $\theta = \phi_m$ for some $m < i^*$ and if $\tup l \theta\in\Gamma$, then $\theta = \psi_m$ for some $m< i^*$.

\begin{claim}
  For each $\psi$ with $\psi\in\sub{\phi}$, if $i >i^*$, then for each $y\in T_{i}$, we have $T_i,y\Vdash \psi$ if and only if $B,y\Vdash \psi$.
\end{claim}
\begin{proof}[Proof of the Claim]
  By induction. Propositional variables, negations of propositional variables and the Booleans are trivial. Assume now that $\psi=[\ell]\chi$. Recall \ref{eq:rel-closure} and \ref{eq:Gamma-i}.
  \begin{itemize}
    \item $y\notin R_\ell^B$, it means for all $z$ and for all $j$, $(y,z)\notin Q_\ell^j$. Then $y<_\ell^B z$ if and only if $z\in T_i$ and $y<_\ell^{T_i}z$, then by the induction hypothesis for $\chi$, we have $T_i,y\Vdash[\ell]\chi$ if and only if $B,y\Vdash [\ell]\chi$;
  \item $y\in R_\ell^B$. Note that $y\in R_\ell^B$ and $y\in T_i$ implies $n_i<\ell$. If $\ell\in N^\bot$, then $T_i,y\Vdash [\ell]\bot$ and $B,y\Vdash [\ell]\bot$. If $\ell\in N^\top$, then $y$ has infinitely many $\ell$-successors in $B$;
  \begin{itemize}
    \item if $[\ell]\chi\in\Gamma$, then $[\ell]\chi,[\ell][m]\chi, [m]\chi\in \Gamma(j)$ for each $j>i^*$  and $\ell<m\le m_i$. It follows that $T_i,y\Vdash [\ell]\chi$. Furthermore, for each $j>i$, $R_{\ell-1}^{T_{j}}\subset v(\chi)$, it follows that for all but finitely many $z\in B$, $y<_\ell^B z$ implies $B,z\Vdash\chi$ and so $B,y\Vdash [\ell]\chi$ as well;
    \item if $\tup{\ell}\neg\chi\in\Gamma$, then for each $j \ge i^*$ with $n_{j}>\ell$, $\tup\ell\neg\chi\in\Gamma(j)$ and $T_j,r^{T_j}\Vdash \Gamma(j)$ implies that there is $z_j\in R_{\ell-1}^{T_j}$ with $T_j,z_j\Vdash\neg\chi$. Recall that $y<_\ell ^B z$ for $z\in T_{i'}$ if and only if $n_{i'}\ge \ell$ and $z\in R^{T_i}_{\ell-1}\setminus R^{T_i}_\ell$. It follows that $B,y\Vdash\tup\ell\neg\chi$ witnessed by the set $\set{z_j: j>i\text{ and }n_j>\ell}$, which is infinite. At the same time $n_i<\ell$ and $i>i^*$ so $y\in R_\ell^{T_i}$ implies $T_i,y\Vdash \tup\ell\neg\chi$.
  \end{itemize}
\end{itemize}
\end{proof}
Now \eqref{eqJStrongCompInductionBox} follows easily. If $[k]\psi\in\Gamma$, then $\psi,[k]\psi\in \Gamma(i)$ for all $i>i^*$. Then $T_i,r^{T_i}\Vdash [k]\psi\land\psi$ for each $i>i^*$ and so $B,r^{T_i}\Vdash[\ell]\psi\land\psi$, it follows that $B,r\Vdash [\ell]\psi$. The same argument for $\tup k\psi$ completes the other direction.

We showed \eqref{eqJStrongCompInductionBox}, which completes the proof of the lemma.\end{proof}

This finishes the proof of Theorem~\ref{thm:bouquet-n}. Indeed, we showed that for any $\Gamma$ consistent with $\mathsf{J}$, there is a $\mathsf{J}$-bouquet with root $r$ and 
valuation $v:\vars\to P(B)$ such that $B,r\Vdash \Gamma$. Thus, we have the following strong completeness result:
\begin{theorem}
$\mathsf{J}$ is strongly complete with respect to $\mathsf{J}$-spaces, i.e., if $\Gamma$ is a set of formulae consistent with $\mathsf{J}$, then there are a  $\mathsf{J}$-space $(X, \vec \tau)$, a point $x \in X$, and a topological model based on $(X, \vec \tau)$ such that 
\[X,x\Vdash \Gamma.\]
\end{theorem}

As an immediate corollary, we obtain the following strong completeness theorem for $\glp$:

\begin{corollary}\label{CorollaryGLPStrongComp}
Suppose $\Gamma$ is a set of formulae consistent with $\glp$. Then there are a  $\mathsf{J}$-space $(X, \vec \tau)$, a point $x \in X$, and a topological model based on $(X, \vec \tau)$ such that
\[X, x\Vdash \mathsf{GLP} + \Gamma.\]
\end{corollary}

\section{Concluding remarks}\label{sec:conclusion}

Beklemishev and Gabelaia \cite{BG13} used the Kripke completeness of $\mathsf{J}$ as an ingredient in their proof of the topological completeness of $\glp$. It is natural to conjecture, in light of the strong completeness result, that $\glp$ is strongly complete with respect to the spaces employed by Beklemishev and Gabelaia, or that the closed fragment of $\glp$ is strongly complete with respect to the spaces introduced by Icard \cite{Ic08}. Recall that the Icard space $\mathcal{I}_n$ on an ordinal $\Theta$ is defined as the topology generated by all intervals
\[(\alpha,\beta]_m = \{x <\Theta: \alpha < \log^m(x) \leq \beta\} \quad \text{ for $m<n$},\]
where the $\log$ function is defined by $\log(\zeta + \omega^\xi) = \xi$. Beklemishev-Gabelaia spaces are rank-preserving extensions of Icard spaces which are maximal with the property that no new neighbourhoods are added around points of successor rank; we refer the reader to \cite{BG13} for more details. The following counterexample has previously appeared in the conference paper \cite{AS24}, but we repeat it for the reader's convenience. Below, we use $\glp^0$ to denote the closed fragment of $\glp$.

\begin{theorem}\label{TheoremCounter}
$\glp^0$ is not strongly complete with respect to any Icard space. $\glp$ is not strongly complete with respect to any Beklemishev-Gabelaia space.
\end{theorem}
\begin{proof}
We let 
\[\Gamma=\set{\tup{0}^k \top : k<\omega}\cup\set{[1]\bot}\cup\set{[0][1]\bot}.\]
This set is consistent with $\glp$. To see this, note that each finite subset is valid in a space of the form $([0,\omega^k], \vec\tau)$, where $\tau_0$ is the order topology and each further $\tau_i$ is discrete.

Suppose towards a contradiction that $\alpha\Vdash \Gamma$ for some ordinal $\alpha$ in some Icard space (the proof for Beklemishev-Gabelaia spaces is the same).
Write $\alpha$ in the form
\[\alpha = \beta + \omega^{\gamma}.\]
Since $\alpha \Vdash [0][1] \bot$, it follows that all ordinals $\alpha'$ sufficiently close to $\alpha$ satisfy $\log^2(\alpha') = 0$, which implies $\gamma = \log(\alpha) \leq \omega$.
Since $\alpha\Vdash[1]\bot$, we must have $\log^2(\alpha) = 0$,  so $\gamma$ is finite. Hence, $\alpha$ is of the form $\beta + \omega^k$ for some $k\in\mathbb{N}$ and thus satisfies $[0]^{k+1}\bot$, which is a contradiction.
\end{proof}

\end{document}